\documentclass[12pt,reqno]{amsart}
\usepackage{amsmath,amssymb,amsthm,calc,verbatim,tikz,url,hyperref,cite,fullpage}
\usepackage{bbm}
\usepackage{setspace}
\usepackage{amssymb}
\usepackage{amsthm}
\usepackage{amsmath}
\usepackage{latexsym}
\usepackage{color}

\renewcommand\labelenumi{(\arabic{enumi})}
\renewcommand\theenumi\labelenumi

\newtheoremstyle{case}{}{}{\normalfont}{}{\itshape}{:}{ }{}
 
\newtheorem{thm}{Theorem}[section]
\newtheorem{lemma}[thm]{Lemma}

\newtheorem{obs}[thm]{Observation}

\theoremstyle{definition}

\newtheorem{alg}[thm]{Algorithm}

\newtheorem*{qu*}{Question}
\theoremstyle{remark}

\newtheoremstyle{case}{}{}{\normalfont}{}{\itshape}{\normalfont:}{ }{}

\theoremstyle{case}

\numberwithin{thm}{section}

\newcommand{\rank}{\textnormal{rk}}

\newcommand{\adj}{\textnormal{adj}}

\newcommand\N{\mathbb{N}}
\newcommand\R{\mathbb{R}}
\newcommand\Z{\mathbb{Z}}
\newcommand\F{\mathbb{F}}
\newcommand\Ex{\mathbb{E}}

\newcommand\cC{\mathcal{C}}

\newcommand\cE{\mathcal{E}}

\newcommand\cU{\mathcal{U}}
\newcommand\cV{\mathcal{V}}
\newcommand\cX{\mathcal{X}}

\renewcommand\Pr{\mathbb{P}}

\renewcommand\leq{\leqslant}
\renewcommand\geq{\geqslant}
\renewcommand\le{\leqslant}
\renewcommand\ge{\geqslant}
\renewcommand\to{\rightarrow}
\newcommand\ds{\displaystyle}

	\def\R{\mathbb{R}}
	\def\Z{\mathbb{Z}}

	\def\N{\mathbb{N}}

	\def\0{\mathbf{0}}

	\def\<{\langle }
	\def\>{\rangle }

\begin{document}

\title{On the singularity of random symmetric matrices}
\author{Marcelo Campos \and Let\'icia Mattos \and Robert Morris \and Natasha Morrison}

\address{IMPA, Estrada Dona Castorina 110, Jardim Bot\^anico, Rio de Janeiro, 22460-320, Brazil}\email{marcelo.campos|leticiamat|rob@impa.br}

\address{Department of Mathematics and Statistics, University of Victoria, David Turpin Building, 3800 Finnerty Road, Victoria, B.C., Canada V8P 5C2} \email{nmorrison@uvic.ca}

\thanks{The first and fourth authors were supported by CNPq, the second author was supported by CAPES, and the third author was partially supported by CNPq (Proc.~303275/2013-8) and FAPERJ (Proc.~201.598/2014).}

\begin{abstract}
A well-known conjecture states that a random symmetric $n \times n$ matrix with entries in $\{-1,1\}$ is singular with probability $\Theta\big( n^2 2^{-n} \big)$. In this paper we prove that the probability of this event is at most $\exp\big( - \Omega( \sqrt{n} ) \big)$, improving the best known bound of $\exp\big( - \Omega( n^{1/4} \sqrt{\log n} ) \big)$, which was obtained recently by Ferber and Jain. The main new ingredient is an inverse Littlewood--Offord theorem in $\Z_p^n$ that applies under very mild conditions, whose statement is inspired by the method of hypergraph containers.
\end{abstract}

\maketitle

\section{Introduction}

{\setstretch{1.12}

Let $A_n$ denote a (uniformly-chosen) random $n \times n$ matrix with entries in the set $\{-1,1\}$. An old and notorious conjecture (see, for example, the discussion in~\cite{KKSz}) states that the probability that $\det(A_n) = 0$ is asymptotically equal to the probability that two of the rows or columns of $A_n$ are equal (up to a factor of $\pm 1$), and hence is equal to $\big(1 + o(1) \big) n^2 2^{-n+1}$. The~first progress on this conjecture was made in 1967, by Koml\'os~\cite{K67}, who used Erd\H{o}s' celebrated solution~\cite{E45} of the Littlewood--Offord problem (see below) to deduce that $A_n$~is singular with probability at most $O(n^{-1/2})$. However, the first exponential bound on the probability was only obtained in 1995, by Kahn, Koml\'os and Szemer\'edi~\cite{KKSz}. Following improvements in the late 2000s by Tao and Vu~\cite{TV07} and by Bourgain, Vu and Wood~\cite{BVW}, a major breakthrough was made recently by Tikhomirov~\cite{T18}, who proved that
$$\Pr\big( \det(A_n) = 0 \big) = \bigg( \frac{1}{2} + o(1) \bigg)^n.$$

In this paper we will consider the analogous problem for random \emph{symmetric} $\pm 1$ matrices, for which significantly less is known. As in the case of $A_n$, it is natural to conjecture that such a matrix is singular with probability $\Theta\big( n^2 2^{-n} \big)$; however, it turns out to be extremely difficult even to prove that $M_n$ is invertible with high probability as $n \to \infty$. The latter problem was apparently first posed by Weiss in the early 1990s (see~\cite{CTV}), but only resolved in 2005, by Costello, Tao and Vu~\cite{CTV}, who proved that 
\begin{equation}\label{eq:CTV:bound}
\Pr\big( \det(M_n) = 0 \big) \le n^{-1/8 + o(1)},
\end{equation}
where we write $M_n$ for a (uniformly-chosen) random $n \times n$ symmetric matrix with entries in the set $\{-1,1\}$. The first super-polynomial bound on the probability that $M_n$ is singular, and the first exponential-type bound (i.e., of the form $\exp( -n^c )$ for some $c > 0$), were obtained almost simultaneously, by Nguyen~\cite{N12} and Vershynin~\cite{V14}, respectively. We remark that the proof in~\cite{N12} was based on earlier work of Nguyen and Vu~\cite{NV11}, which relied on deep results from additive combinatorics, while the proof in~\cite{V14} built on the earlier breakthroughs of Rudelson and Vershynin~\cite{RV08,RV09}.

Recently, a new `combinatorial' approach to studying the invertibility of discrete random matrices was introduced by Ferber, Jain, Luh, and Samotij~\cite{FJLS}, and applied by Ferber and Jain~\cite{FeJ} to prove that 
$$\Pr\big( \det(M_n) = 0 \big) \le \exp\big( - c n^{1/4} \sqrt{\log n} \big)$$
for some $c > 0$. In this paper we use a different combinatorial approach (inspired by the method of~\cite{FJLS,FeJ}) to obtain the following bound.

\begin{thm}\label{thm:main}
There exists $c > 0$ such that 
\begin{equation}\label{eq:main:thm:bound}
\Pr\big( \det(M_n) = 0 \big) \le \exp\big( - c \sqrt{n} \big)
\end{equation}
for all sufficiently large $n \in \N$.
\end{thm}

} 
{\setstretch{1.13}

The main new ingredient in our approach is an inverse Littlewood--Offord theorem (see Theorem~\ref{thm:containers}, below) which applies to vectors $v \in \Z_p^n$ that exhibit a very mild amount of `structure'. 
In order to motivate this theorem, let us begin by recalling the problem of Littlewood and Offord~\cite{LO43}, introduced in 1943 during their study of random polynomials. For any abelian group $G$, integer $n \in \N$, and vector $v \in G^n$, define
$$\rho(v) 
:= \max_{a \in G} \Pr\bigg( \sum_{i =1}^n u_i v_i = a \bigg),$$
where $u$ is a uniformly-chosen random element of $\{-1,1\}^n$. Littlewood and Offord~\cite{LO43} proved that $\rho(v) = O\big( |v|^{-1/2} \log |v| \big)$ when $G = \Z$, where $|v| := \big| \big\{ i \in [n]: v_i \neq 0 \big\} \big|$ denotes the size of the support of~$v$. Erd\H{o}s~\cite{E45} improved this to $\rho(v) = O\big( |v|^{-1/2} \big)$, which is best possible, using Sperner's theorem. The problem of obtaining upper bounds on $\rho(v)$ (and other related functions) has become known as the `Littlewood--Offord problem', and has been extensively studied over the past several decades, for example by Rogozin~\cite{R61}, S\'ark\H{o}zy and Szemer\'edi~\cite{SS65}, Esseen~\cite{E66}, Hal\'asz~\cite{H77}, and Frankl and F\"uredi~\cite{FF88}. More recently, Costello, Tao and Vu~\cite{CTV} proved a `quadratic' Littlewood--Offord inequality, and used it to deduce their bound~\eqref{eq:CTV:bound} on the probability that $M_n$ is singular. 

Inverse Littlewood--Offord theory, the study of the structure of vectors $v \in G^n$ such that $\rho(v)$ is (relatively) large, was initiated by Tao and Vu~\cite{TV09a}, and has since played an important role in the study of random matrices, see for example the seminal work of Rudelson and Vershynin~\cite{RV08,RV09}, Tao and Vu~\cite{TV09b}, Nguyen and Vu~\cite{N12,NV11}, and the surveys~\cite{NV13,RV10,Vu14}. Our inverse Littlewood--Offord theorem differs from (most of) these earlier results in several important ways: it is designed for $\Z_p$, rather than $\Z$; it gives (weak) structural information about every vector $v \in \Z_p^n$ such that $\rho(v) \ge 4/p$; and it is designed to facilitate iteration. In particular, the method of Nguyen and Vu (see, e.g.,~\cite[Theorem~8.1]{NV13}) requires that $\rho(v) \ge n^{-C}$, and those of Ferber, Jain, Luh, and Samotij~\cite[Theorem~1.7]{FJLS} and of Rudelson and Vershynin (see, e.g.,~\cite[Theorem~1.5]{RV08}) seem difficult to iterate. Let us note, however, that the powerful results of Rudelson and Vershynin~\cite{RV08,RV09} are applicable for $\rho(v) \ge e^{-cn}$, and are moreover valid over the real numbers. We remark that the statement of Theorem~\ref{thm:containers} was inspired by the method of hypergraph containers, a technique that was introduced several years ago by Balogh, Morris and Samotij~\cite{BMS} and (independently) Saxton and Thomason~\cite{ST}, and which has turned out to have a large number of applications in extremal and probabilistic combinatorics. We refer the interested reader to the survey~\cite{BMS18} for more details. 

Given a vector $v \in \Z_p^n$ and a subset $Y \subset [n]$, let us write $v_Y$ for the restriction of $v$ to the coordinates of $Y$. Recall that $|v|$ denotes the size of the support of~$v$. 
Our inverse Littlewood--Offord theorem is as follows.

\begin{thm}\label{thm:containers}
Let $p$ be a prime. There exists a family\/ $\cC$ of subsets of\/ $\Z_p$, with
\begin{equation}\label{eq:number:of:containers}
|\cC| \le \exp\Big( 2^{12} \big( \log p \big)^2 \Big),
\end{equation}
such that for each $n \in \N$, and every $v \in \Z_p^n$ with $\rho(v) \ge 4/p$ and $|v| \ge 2^{18} \log p$, there exist sets $B(v) \in \cC$ and $Y = Y(v) \subset [n]$, with $n/4 \le  |Y| \le n/2$, such that
\begin{equation}\label{eq:containers:properties}
\big| \big\{ i \in [n] : v_i \notin B(v) \big\} \big| \le \frac{n}{4} \qquad \text{and} \qquad |B(v)| \le \frac{2^{16}}{ \rho(v_Y) \sqrt{|v|}}.
\end{equation}
\end{thm}

In order to motivate the statement of the theorem above, it is instructive to consider the example of a vector $v \in \Z_p^n$ whose entries are chosen uniformly (and independently) at random from a $d$-dimensional generalised arithmetic progression\footnote{This is a set of the form $\big\{ a + j_1 \ell_1 + \cdots + j_d \ell_d : 1 \le j_i \le k_i \big\}$ for some $a,\ell_1,\ldots,\ell_d \in \Z_p$ and $k_1,\ldots,k_d \in \N$.} $Q$. For such a vector, $\rho(v)$ is typically of order $|Q|^{-1} n^{-d/2}$ (since in each dimension the random walk spreads out by a factor of roughly $\sqrt{n}$), and the $p^{\Theta(d)}$ such progressions are natural `containers' for these vectors. This example suggests that one might be able to prove a stronger version of Theorem~\ref{thm:containers}, in which most `containers' (members of the family $\cC$) are significantly smaller than the maximum given in~\eqref{eq:containers:properties}. However, without substantial additional ideas such a strengthening would \emph{not} imply a significant improvement over the bound in Theorem~\ref{thm:main}, see the discussion in Section~\ref{barrier:sec} for more details.

We remark that the sets $Y(v)$, whose appearance in Theorem~\ref{thm:containers} might appear somewhat unnatural at first sight, will play a vital role in our application of the theorem to prove Theorem~\ref{thm:main}. More precisely, we will use the sets $Y(v)$ to maintain independence as we reveal various rows and columns of the matrix, see Section~\ref{sketch:app:sec} for more details. Let us also mention here that the family of containers $\cC$ will be defined explicitly (see~\eqref{def:C}, below), but we will only need the properties stated in the theorem. The proof of Theorem~\ref{thm:containers} uses the probabilistic method (for those readers familiar with the container method, we choose the `fingerprint' randomly), and a classical  `anticoncentration lemma' proved by Hal\'asz~\cite{H77} (Lemma~\ref{lem:halasz}, below), see Section~\ref{sketch:containers:sec} for more details.

The rest of the paper is organised as follows: in Section~\ref{sec:overview} we give an overview of the proof, in Section~\ref{sec:containers} we prove Theorem~\ref{thm:containers}, and in Section~\ref{sec:iterate} we deduce Theorem~\ref{thm:main}. 
}

\section{An overview of the proof}\label{sec:overview}

In this section we will outline the proof of our inverse Littlewood--Offord theorem, and the deduction of Theorem~\ref{thm:main}. The first step is to apply the method of~\cite{CTV,FeJ,N12} to reduce the problem to bounding the quantity
\begin{equation}\label{eq:qdef}
q_n(\beta) := \max_{w \in \Z_p^n} \Pr\Big( \exists \, v \in \Z_p^n \setminus \{0\}: M_n \cdot v = w \text{ and } \rho(v) \ge \beta \Big),
\end{equation}
for some suitable $\beta = \exp\big( - \Theta( \sqrt{n} ) \big)$ and a prime $p = \Theta\big( 1 / \beta \big)$. To be precise, we will use the following lemma, which was proved by Ferber and Jain~\cite{FeJ} using techniques developed by Costello, Tao and Vu~\cite{CTV} and Nguyen~\cite{N12}. Note that the dependence of $q_n(\beta)$ on the prime $p$ is suppressed in the notation. 

\begin{lemma}\label{lem:sufftoshow}
Let $n \in \N$, and let $p > 2$ be prime. For every $\beta > 0$, 
$$\Pr\big( \det(M_n) = 0 \big) \le \, 16n \sum_{m = n-1}^{2n-3} \bigg( \beta^{1/8} + \frac{q_m(\beta)}{\beta} \bigg).$$
\end{lemma}
 
Since Lemma~\ref{lem:sufftoshow} was not stated explicitly in~\cite{FeJ}, for completeness we provide the proof in Appendix~\ref{app:red}. Using our inverse Littlewood--Offord theorem (Theorem~\ref{thm:containers}), we will prove the following bound on $q_n(\beta)$. 

\begin{lemma}\label{thm:qrho}
Let $n \in \N$ be sufficiently large, and let\/ $p \le \exp\big( 2^{-10} \sqrt{n} \big)$ be prime. Then 
$$q_n(\beta) \le 2^{-n/4}$$
for every\/ $\beta \ge 4/p$.
\end{lemma}

Theorem~\ref{thm:main} is easily deduced from Lemmas~\ref{lem:sufftoshow} and~\ref{thm:qrho}.

\begin{proof}[Proof of Theorem~\ref{thm:main}, assuming Lemmas~\ref{lem:sufftoshow} and~\ref{thm:qrho}]
Let $n \in \N$ be sufficiently large, let\/ $\exp\big( 2^{-11} \sqrt{n} \big) \le p \le 2 \cdot \exp\big( 2^{-11} \sqrt{n} \big)$ be prime, and set $\beta := 4/p$. By Lemmas~\ref{lem:sufftoshow} and~\ref{thm:qrho}, it follows that
$$\Pr\big( \det(M_n) = 0 \big) \le \, 16n \sum_{m = n-1}^{2n-3} \bigg( \big( 4/p \big)^{1/8} + \frac{p}{2^{m/4 + 2}} \bigg) \le \exp\big( - c \sqrt{n} \big)$$
for some $c > 2^{-15}$, as required.
\end{proof}

We will prove Theorem~\ref{thm:containers} in Section~\ref{sec:containers}, and deduce Lemma~\ref{thm:qrho} in Section~\ref{sec:iterate}. Although the proofs are not especially technical, some of the definitions may initially seem somewhat surprising. In order to motivate these definitions, we will now provide a brief outline of the argument, beginning with the deduction of Lemma~\ref{thm:qrho} from Theorem~\ref{thm:containers}. 


\subsection{An outline of the proof of Lemma~\ref{thm:qrho}}\label{sketch:app:sec}

We will bound $q_n(\beta)$ using the first moment method: for each $w \in \Z_p^n$, we will bound the expected number of vectors $v \in \Z_p^n \setminus \{0\}$ with $\rho(v) \ge \beta$ such that $M_n \cdot v = w$. In order to do so, we will use Theorem~\ref{thm:containers} to partition the collection of vectors $v \in \Z_p^n$ with $\rho(v) \ge \beta$ and $|v| \ge \lambda\sqrt{n}$ into a collection $\cU$ of at most $n^{cn}$ `fibres' (for some $\lambda > 0$ and $c > 0$); we will then apply the union bound inside each fibre. 
The bound we obtain on the probability that $M_n \cdot v = w$ will depend on the fibre containing $v$, and the partition is chosen so that (for each $S \in \cU$ and $w \in \Z_p^n$) the expected number of vectors $v \in S$ with $M_n \cdot v = w$ is at most $n^{-c'n}$ (for some $c' > c$). The claimed bound then follows by summing over fibres, and then dealing with the vectors with small support separately via a simple counting argument (see Lemma~\ref{lem:small:support}). 

To construct the partition, we need to map each vector $v \in \Z_p^n \setminus \{0\}$ to a fibre containing $v$. To do so, we repeatedly apply Theorem~\ref{thm:containers} to vectors of the form $v_X$, for some set $X \subset [n]$ given by the earlier steps of the process. The theorem provides us with a container $B(v_X)$ for $v_X$, and this allows us to bound the number of choices for $v_X$. Revealing the rows of $M_n$ corresponding to $X$, we will be able to use the probability that $M_{X \times [n]} \cdot v = w_X$, and the bound on $|B(v)|$ given by~\eqref{eq:containers:properties}, to `beat' this number of choices. We continue this iteration until we have chosen all but $O(\sqrt{n})$ of the non-zero entries of $v$. 

To describe a single step of this iteration, assume that we have already revealed a subset of the rows of $M_n$, and let $Z \subset [n]$ denote the set of rows that have not yet been revealed. By Theorem~\ref{thm:containers}, we may associate, to each vector $v \in \Z_p^n \setminus \{0\}$ with $\rho(v_Z) \ge \rho(v) \ge \beta \ge 4/p$ (see Observation~\ref{obs:rho:monotone}, below) and $|v_Z| \ge 2^{18} \log p$, sets
$$Y(v_Z) \subset Z, \quad B(v_Z) \subset \Z_p \quad \text{and} \quad X(v_Z) := \big\{ i \in Z \setminus Y(v_Z) : v_i \in B(v_Z) \big\}.$$
In this step we will `reveal' the rows of $M_n$ corresponding to $X = X(v_Z)$, and sum over the choices for $v_i \in B(v_Z)$ for each $i \in X$. We claim that 
\begin{equation}\label{eq:MXY:sketch}
\Pr\big( M_{X \times [n]} \cdot v = w_X \big) \le \rho(v_Y)^{|X|}.
\end{equation}
Indeed, since $X$ and $Y = Y(v_Z)$ are disjoint subsets of $Z$, the entries of $M_{X \times Y}$ are all independent (of each other, and of the previously revealed entries of $M_n$), so the claimed bound holds by the definition of $\rho$ (see the proof of Lemma~\ref{lem:key:calc}, below, for the details). 

\enlargethispage*{\baselineskip}

Each fibre will be the set of vectors that have the same sequence of sets $X$, $Y$ and $B(v_Z)$, and therefore in order to bound the number of fibres, we just need to count the number of choices for these sets. 
We have at most $2^{|Z|}$ choices each for $X$ and $Y$, and at most  
$$\exp\Big( 2^{12} \big( \log p \big)^2 \Big) \le \exp\big( 2^{-8} n \big)$$
choices for the set $B(v_Z)$, by~\eqref{eq:number:of:containers} and our choice of $p$. Now, it follows from~\eqref{eq:containers:properties} and our bounds on $|Y|$ that $|X| \ge |Z|/4$, and hence the total number of choices for these sets (over all steps of the process) is at most $\exp\big( 2^{-6} n \log n \big)$, see Lemma~\ref{lem:Usize}, below. 

Finally, we have at most $|B(v_Z)|^{|X|}$ choices for the vector $v_X$. Multiplying this by the probability bound~\eqref{eq:MXY:sketch}, and using the bound on $|B(v_Z)|$ given by~\eqref{eq:containers:properties}, we obtain  
$$|B(v_Z)|^{|X|} \rho(v_Y)^{|X|} \le \bigg( \frac{2^{16}}{\sqrt{|v_Z|}} \bigg)^{|X|} \le n^{-|X|/4},$$
since $|v_Z| \ge \lambda\sqrt{n}$. Since $|X| \ge n/4$ in the first step, this will be sufficient to prove the claimed bound on the expected number of vectors $v \in C$ with $M_n \cdot v = w$. 

Before continuing, let us briefly discuss why we are unable to use Theorem~\ref{thm:containers} to prove a stronger bound than that in Theorem~\ref{thm:main}. Recall that our probability bound for a given fibre is of the form $n^{-cn}$, and the number of fibres is roughly $\exp\big( (\log p)^2 \log n \big)$, since we need to iterate the process described above $\log n$ times. (We remark that for this reason we also cannot deduce a stronger bound for the probability that $\det(A_n) = 0$ from Theorem~\ref{thm:containers}.) In order to improve our bound further, one would therefore need to either find a smaller set of containers, or find a set of containers that covers much more of the vector than those in Theorem~\ref{thm:containers}. It is possible that an improvement of this type could be used to prove a bound of the form $\exp\big( - \sqrt{n\log n} \big)$; however, as we will show next, in order to obtain any further improvement significant new ideas would be needed.

\subsection{A natural barrier at $\exp\big( - \sqrt{n\log n} \big)$}\label{barrier:sec}

In this section we explain why a simple union bound (like that described in Section~\ref{sketch:app:sec}) cannot be used to prove a significantly stronger bound than that in Theorem~\ref{thm:main}, without `reusing' some of the randomness in $M_n$. Let $m \le n$, and consider the family of vectors $v \in \Z^n$ whose entries are chosen from the set $\{-N,\ldots,N\}$, where $N = c n^{-1/2} 2^m$ (for some small $c > 0$). For any such $v$, we have
$$\rho(v_{[k]}) \ge \rho(v) \ge 2^{-m}$$
for every $k \ge m$, since a random walk with step sizes at most $N$ ends in the interval $[-2^{m-2},2^{m-2}]$ with probability at least $1/2$. Note also that $\rho(v_{[k]}) \ge 2^{-k}$ for every $k < m$.

Now, it follows that the natural bound\footnote{Note that we are losing something here, since the already-revealed part of the matrix may map $v_{[k+1,n]}$ to a non-maximiser of $\rho$. Controlling this (and using an anticoncentration inequality that is sensitive to the choice of the maximiser) may be a way to overcome the barrier described in this section.} 
$$\Pr\big( M_n \cdot v = 0 \big) \le \, \prod_{k=1}^n \rho(v_{[k]}),$$
which uses all of the randomness in $M_n$, cannot give a stronger bound than
$$\Pr\big( M_n \cdot v = 0 \big) \le 2^{-m(n - m)} \prod_{k=1}^m 2^{-k} = 2^{-mn + m^2/2 + O(n)}.$$
Since there are $(2N+1)^n \ge c^n 2^{mn} n^{-n/2}$ choices for the vector $v$, a union bound (over these vectors) gives (at best) a bound of $n^{-n/2} 2^{m^2/2 + O(n)}$, which is small only if $m \le \sqrt{n \log n}$.

It follows that our proof method only has a chance of working if $p \le \exp\big( \sqrt{n \log n} \big)$. However, if we are working over $\Z_p$ then we cannot hope to prove a much stronger bound on the singularity probability than $1/p$. Indeed, let $M_{n-1}$ be the matrix obtained by removing the first row and column of $M_n$, and let $u \in \{-1,1\}^{n-1}$ be obtained from the first row of $M_n$ by deleting the entry $m_{11}$. Now, if $\det(M_{n-1}) \ne 0$ and $\< u, \, M_{n-1}^{-1} \cdot u \> = m_{11}$, then there exists a vector $w := (1,-M_{n-1}^{-1} \cdot u) \in \Z_p^n \setminus \{0\}$ with $M \cdot w = 0$, and hence $\det(M_n) = 0$. It follows that the event that $\det(M_n) = 0$ is (up to a constant factor) at least as likely as the event that $\< u, \, M_{n-1}^{-1} \cdot u \> \in \{-1,1\}$, and it seems reasonable to expect that this latter event occurs with probability at least $\Omega(1/p)$. 



\subsection{Hal\'asz's inequality, and the inverse Littlewood--Offord theorem}\label{sketch:containers:sec}

In this section we will state the main tool we will use in the proof of Theorem~\ref{thm:containers}, a classical Littlewood--Offord theorem due to Hal\'asz~\cite{H77}. We will also prepare the reader for the proof in the next section by providing some motivation for the way we define our family of containers. 

In order to state Hal\'asz's inequality, we need a little preparation. First, let us define multiplication on $\Z_p$ as follows: if $x,y \in \Z_p$, then the product $x \cdot y \in \Z$ is obtained by projecting $x$ and $y$ onto elements of $\{0,1,\ldots,p-1\}$ in the usual way, and then multiplying in $\Z$. 
Let $\| \cdot \|$ denote the distance to the nearest integer, and for each $n \in \N$, prime $p$ and vector $v \in \Z_p^n$, define the \emph{level sets} of $v$ to be 
\begin{equation}\label{def:level:sets}
T_t(v) := \bigg\{ k \in \Z_p \,:\, \sum \limits_{i = 1}^n \left \| \dfrac{k \cdot v_i}{p} \right \|^2 \le t \bigg\},
\end{equation}
for each $t \ge 0$. 

We can now state the lemma of Hal\'asz~\cite{H77}; since we use a slightly different form than is usually stated, for completeness we provide a proof in Appendix~\ref{app:Halasz}. 

\begin{lemma}[Hal\'{a}sz's Anticoncentration Lemma]\label{lem:halasz}
Let $n \in \N$ and $p$ be prime, and let $v \in \Z_p^n \setminus \{0\}$. Then 
$$\rho(v) \le \dfrac{3}{p} + \frac{4 |T_\ell(v)|}{p \sqrt{\ell}} + e^{-\ell}$$
for every $1 \le \ell \le 2^{-6} |v|$.
\end{lemma}

Let us now motivate the way we choose our family of containers, see~\eqref{def:C}, below. The basic intuition, first suggested by Tao and Vu~\cite{TV07,TV09a}, is that if $\rho(v)$ is large, then $v$ should have some arithmetic structure. We think of the elements of the level sets $T_t(v)$ as `frequencies' that correlate with the entries of $v$, and thus encode this arithmetic structure. Following the strategy of Tao and Vu~\cite{TV07} and Nguyen and Vu~\cite{NV11}, we would therefore like to define the container of each `structured' vector using its level sets.

The problem is that we would like a relatively small family of containers, whereas the number of level sets could potentially be very large. The solution is very simple: we consider a random subset $U$ of the coordinates of $v$. We will show that if $|U| \ge 2^{12} \log p$, then $v_U$ still correlates with the frequencies of the level sets of $v$, and we will choose the container of $v$ to be (roughly speaking) the elements of $\Z_p$ that correlate with these frequencies. We then choose $U$ as small as possible (subject to the above argument working), which implies that there are few choices for the vector $v_U$, and hence few containers.

\section{Proof of the inverse Littlewood--Offord theorem}\label{sec:containers}

In this section we will prove Theorem~\ref{thm:containers}. Let $n \in \N$ and a prime $p$ be fixed
throughout the section, and assume that $p > 3$ and $n \ge 2^{18} \log p$ (since otherwise there are no vectors $v \in \Z_p^n$ with $\rho(v) \ge 4/p$ and $|v| \ge 2^{18} \log p$, so the statement holds vacuously with $\cC = \emptyset$). 



For each $m \in \N$ and $w \in \Z_p^m$, define (cf.~\cite[Section~7]{TV07} and~\cite[Section~5]{NV11}) the set of `frequencies' of $w$ to be
$$F(w) := \left\{ k \in \Z_p \,:\, \sum_{i = 1}^m \left\|\frac{k \cdot w_i}{p}\right\|^2 \le \log p \right\},$$
and note (recalling~\eqref{def:level:sets}) that $F(w) = T_{\log p}(w)$. Now, for each $S \subset \Z_p$, define
\begin{equation}\label{def:C:of:S}
C(S):= \bigg\{ a \in \Z_p \,:\, \sum_{k \in S} \bigg\| \frac{a \cdot k}{p} \bigg\|^2 \le \, \frac{|S|}{2^5} \bigg\}.
\end{equation}
Now set $m := \lfloor 2^{12} \log p \rfloor$, and define
\begin{equation}\label{def:C}
\cC := \big\{ C\big( F(w) \big) : w \in \Z_p^m \big\},
\end{equation}
and observe that $|\cC| \le p^m$, as required. We will show that $\cC$ has the desired properties.

The following simple lemma motivates our choice of containers (cf.~\cite[Section~5]{NV11}). 

\begin{lemma}\label{lem:containers:contain:most:of:v}
Let $v \in \Z_p^n$, and let $t \le 2^{-7} n$. If $S \subset T_t(v)$, then 
$$\big| \big\{ i \in [n] : v_i \not\in C(S) \big\} \big| \le \frac{n}{4}.$$
\end{lemma}

\begin{proof}
Let $R = \big\{ i \in [n] : v_i \notin C(S) \big\}$, and observe that, by~\eqref{def:level:sets} and~\eqref{def:C:of:S},
$$\frac{|R| |S|}{2^5} \le \sum_{i \in R} \sum_{k \in S} \bigg\| \frac{k \cdot v_i}{p} \bigg\|^2 \le \sum_{k \in S} \sum_{i = 1}^n \bigg\| \frac{k \cdot v_i}{p} \bigg\|^2 \le t  |S| \le \frac{n |S|}{2^7},$$
so $|R| \le n/4$, as required.
\end{proof}

Later in the proof, we will define $B(v) := C\big( F(v_U) \big)$ for some set $U \subset [n]$ with $|U| \le m$ such that $F(v_U) \subset T_t(v)$ for $t = 2^{-7} n$ (see Lemma~\ref{lemma:def:U}, below). We next turn to bounding the size of our containers; the following lemma (cf.~\cite[Section~5]{NV11}) provides a first step.

\begin{lemma}\label{lem:containers:small:first:bound}
For any set $S \subset \Z_p$, we have
\begin{equation}\label{eq:containers:small:first:bound}
|C(S)| \le \frac{4p}{|S|}.
\end{equation}
\end{lemma}

\begin{proof}
We will instead bound the size of the larger set 
$$C'(S):= \bigg\{ a \in \Z_p \,:\, \sum_{k \in S}\cos\bigg( \frac{2\pi a k}{p} \bigg) \ge \frac{|S|}{2} \bigg\}.$$
Indeed, observe that $C(S) \subset C'(S)$, since we have $1 - 2^4 \|x\|^2 \le \cos (2\pi x)$ for every $x \in \R$. 

Now, let $a$ be a uniformly-chosen random element of $\Z_p$, and observe that, by Markov's inequality,
\begin{align*}
\Pr\big( a \in C'(S) \big) & \, = \, \Pr\bigg( \bigg( \sum_{k \in S} \cos\bigg( \frac{2\pi a k}{p} \bigg) \bigg)^2 \ge \frac{|S|^2}{4} \bigg) \\
& \, \le \, \frac{4}{|S|^2} \cdot \frac{1}{p} \sum_{a \in \Z_p} \bigg( \sum_{k \in S} \cos\bigg( \frac{2\pi a k}{p} \bigg) \bigg)^2,
\end{align*}
Now, since $2\cos(x) = e^{ix} + e^{-ix}$, we have
$$4 \sum_{a \in \Z_p} \bigg( \sum_{k \in S} \cos\bigg( \frac{2\pi a k}{p} \bigg) \bigg)^2 = \sum_{k_1 \in \pm S} \sum_{k_2 \in \pm S} \sum_{a \in \Z_p} \exp\bigg( \frac{2\pi i a (k_1 + k_2)}{p} \bigg) \le 4p |S|,$$
where $\pm S$ is the multi-set obtained by taking the union of $S$ and $-S$, counting elements in both twice. For the second step, simply note that the roots of unity sum to zero, so the only terms that contribute are those with $k_1 + k_2 = 0$. It follows that 
$$\frac{4}{|S|^2} \cdot \frac{1}{p} \sum_{a \in \Z_p} \bigg( \sum_{k \in S} \cos\bigg( \frac{2\pi a k}{p} \bigg) \bigg)^2 \le \frac{4}{|S|},$$
and hence $|C(S)| \le |C'(S)| \le 4p / |S|$, as claimed.
\end{proof}

We will use Hal\'{a}sz's Anticoncentration Lemma (Lemma~\ref{lem:halasz}) to bound the right-hand side of~\eqref{eq:containers:small:first:bound} in terms of $\rho(v_Y)$ (for some set $Y$ that will be chosen in Lemma~\ref{lem:Ydef}, below). The following lemma is a straightforward application of Lemma~\ref{lem:halasz}.

\begin{lemma}\label{lem:Halasz:application}
Let $v \in \Z_p^n$ with $\rho(v) \ge 4/p$ and $|v| \ge 2^{18} \log p$, and let $Y \subset [n]$ be such that $|v_Y| \ge |v| / 4$. Then
$$\rho(v_Y) \le \frac{2^{13} |T_\ell(v_Y)|}{p\sqrt{|v|}},$$
where $\ell := 2^{-16} |v|$. 
\end{lemma}

In the proof of Lemma~\ref{lem:Halasz:application}, and also later in the section, we will need the following simple observation (see Lemma~A.10 or~\cite[Lemma~2.8]{FeJ}).

\begin{obs}[Lemma~2.8 of \cite{FeJ}]\label{obs:rho:monotone}
$\rho(v_Y) \ge \rho(v)$ for every\/ $v \in \Z_p^n$ and every\/ $Y \subset [n]$.
\end{obs}

\begin{proof}[Proof of Lemma~\ref{lem:Halasz:application}]
Applying Lemma~\ref{lem:halasz} to $v_Y$, with $\ell = 2^{-16} |v| \le 2^{-14} |v_Y|$, gives 
$$\rho(v_Y)\leq \frac{3}{p} + \frac{4 |T_\ell(v_Y)|}{p\sqrt{\ell}} + e^{-\ell}.$$ 
Now, by Observation~\ref{obs:rho:monotone} and our assumption on $\rho(v)$, we have $\rho(v_Y) \ge \rho(v) \ge 4/p$. Since $\ell \ge 4 \log p$, it follows that
$$\rho(v_Y) \,\le\, \frac{2^5 |T_\ell(v_Y)|}{p\sqrt{\ell}} \, = \, \frac{2^{13} |T_\ell(v_Y)|}{p \sqrt{|v|}},$$ 
as claimed.
\end{proof}

To complete the proof, it will now suffice to choose sets $Y \subset [n]$, with $n/4 \le |Y| \le n/2$, and $U \subset [n]$, with $|U| \le m$, such that 
\begin{equation}\label{eq:YU:conditions}
F(v_U) \subset T_t(v), \qquad |v_Y| \ge \frac{|v|}{4} \qquad \text{and} \qquad |T_\ell(v_Y)| \le 2 \cdot |F(v_U)|,
\end{equation}
where $\ell = 2^{-16} |v|$ and $t = 2^{-7} n$. Indeed, for any such sets we have, by Lemmas~\ref{lem:containers:small:first:bound} and~\ref{lem:Halasz:application},
$$\big| C\big( F(v_U) \big) \big| \le \frac{4p}{|F(v_U)|} \le \frac{2^{15}}{\rho(v_Y) \sqrt{|v|}} \cdot \frac{|T_\ell(v_Y)|}{|F(v_U)|} \le \frac{2^{16}}{\rho(v_Y) \sqrt{|v|}},$$ 
and, by Lemma~\ref{lem:containers:contain:most:of:v}, we have 
$$\big| \big\{ i \in [n] : v_i \not\in C\big( F(v_U) \big) \big\} \big| \le \frac{n}{4}.$$ 
Thus, setting $B(v) := C\big( F(v_U) \big)$, we obtain a set in $\cC$ for which the properties~\eqref{eq:containers:properties} hold.  

We will choose the sets $Y$ and $U$ in the next two lemmas. In each case we simply choose a random set of the correct density. We will say that $R$ is a \emph{$q$-random} subset of a set $S$ if each element of $S$ is included in $R$ independently at random with probability $q$. 

\begin{lemma}\label{lem:Ydef}
Let $v\in \Z_p^n$ with $|v| \ge 2^{18} \log p$. There exists $Y \subset [n]$, with $n/4 \le |Y| \le n/2$, such that
$$|v_Y|\geq \frac{|v|}{4} \qquad \text{and} \qquad T_\ell(v_Y) \subset T_{8\ell}(v),$$
where $\ell = 2^{-16} |v|$.
\end{lemma}

\begin{proof}
Let $Y$ be a $(3/8)$-random subset of $[n]$; we will prove that with positive probability $Y$ has all of the required properties. Since $n \ge |v| \ge 2^{18} \log p \ge 2^{18}$, the properties 
$$\frac{n}{4} \le |Y| \le \frac{n}{2} \qquad \text{and} \qquad |v_Y| \ge \frac{|v|}{4}$$ 
each hold with probability at least $3/4$, by Chernoff's inequality. To bound the probability that $T_\ell(v_Y) \setminus T_{8\ell}(v)$ is non-empty, define a random variable 
$$W(k) := \sum_{i \in Y} \left \| \dfrac{k \cdot v_i}{p} \right \|^2$$
for each $k \in \Z_p$, and observe that, by~\eqref{def:level:sets},
$$k \in T_\ell(v_Y) \;\Leftrightarrow\; W(k) \le \ell \qquad \text{and} \qquad k \not\in T_{8\ell}(v) \;\Rightarrow\; \Ex[W(k)] \ge 3\ell.$$ 
Moreover, by Chernoff's inequality,\footnote{Here we use the following variant of the standard Chernoff inequality: if $X_1,\ldots,X_N$ are iid Bernoulli random variables, and $t_1,\ldots,t_N \in [0,1]$, then $\Pr\big( \sum_{i = 1}^N t_i X_i \le s \big) \le \exp\big( - \Ex[X]/2 + s \big)$.}
$$\Pr\big( k \in T_{\ell}(v_Y) \big) = \Pr\big( W(k) \le \ell \big) \le e^{-\ell/2} \le \frac{1}{p^2}$$
for every $k \not\in T_{8\ell}(v)$, since $\ell \ge 4 \log p$. It follows that
$$\Ex\big[ |T_{\ell}(v_Y) \setminus T_{8\ell}(v)| \big] \le \frac{1}{p},$$
and hence $T_{\ell}(v_Y) \subset T_{8\ell}(v)$ with probability at least $3/4$, as required. 
\end{proof}

Finally, we need to show that a suitable set $U$ exists. 

\begin{lemma}\label{lemma:def:U}
Let $v \in \Z_p^n$. There exists $U \subset [n]$, with $|U| \le m$, such that 
$$|T_{8\ell}(v)| \le 2 \cdot |F(v_U)| \qquad \text{and} \qquad F(v_U) \subset T_t(v),$$
where $\ell = 2^{-16} |v|$ and $t = 2^{-7} n$.
\end{lemma}

\begin{proof}
Let $U$ be a $( m / 2n )$-random subset of $[n]$. We will prove that the claimed properties hold simultaneously with positive probability. Note first that $|U| \le m$ with probability at least $3/4$, by Chernoff's inequality, since $m = \lfloor 2^{12} \log p \rfloor \ge 2^{12}$. 

Next, we show that $|T_{8\ell}(v) \setminus F(v_U)| \le |T_{8\ell}(v)|/2$ with probability at least $1/2$. Observe first that, for every $k \in \Z_p$, 
$$\Pr\big( k \not\in F(v_U) \big) = \Pr\left( \sum_{i \in U} \left\|\frac{k \cdot v_i}{p}\right\|^2 > \log p \right) \le \frac{1}{\log p} \cdot \Ex\left[ \sum_{i \in U} \left\|\frac{k \cdot v_i}{p}\right\|^2 \right],$$
by Markov's inequality. Now, if $k \in T_{8\ell}(v)$, then 
$$\frac{1}{\log p} \cdot \Ex\left[ \sum_{i \in U} \left\|\frac{k \cdot v_i}{p}\right\|^2 \right] = \frac{m}{2n \log p} \sum_{i = 1}^n \left\|\frac{k \cdot v_i}{p}\right\|^2 \le \frac{8m\ell}{2n\log p} \le \frac{1}{4},$$ 
since $m \le 2^{12} \log p$ and $\ell = 2^{-16} |v| \le 2^{-16} n$. It follows that
$$\Pr\bigg( |T_{8\ell}(v) \setminus F(v_U)| \ge \frac{|T_{8\ell}(v)|}{2} \bigg) \le \frac{2}{|T_{8\ell}(v)|} \cdot \Ex\big[ |T_{8\ell}(v) \setminus F(v_U)| \big] \le \frac{1}{2},$$
by Markov's inequality, as claimed.

Finally, to bound the probability that $F(v_U) \setminus T_t(v)$ is non-empty, we repeat the argument used in the proof of Lemma~\ref{lem:Ydef}. To be precise, we define a random variable 
$$W(k) := \sum_{i \in U} \left \| \dfrac{k \cdot v_i}{p} \right \|^2$$
for each $k \in \Z_p$, and observe that, by~\eqref{def:level:sets},
$$k \in F(v_U) \;\Leftrightarrow\; W(k) \le \log p \qquad \text{and} \qquad k \not\in T_t(v) \;\Rightarrow\; \Ex[W(k)] \ge 2^{-8} m.$$ 
Recalling that $m = \lfloor 2^{12} \log p \rfloor$, it follows by Chernoff's inequality that
$$\Pr\big( k \in F(v_U) \big) = \Pr\big( W(k) \le \log p \big) \le \frac{1}{p^2}$$
for every $k \not\in T_t(v)$, and hence
$$\Pr\big( F(v_U) \not\subset T_t(v) \big) \le \Ex\big[ |F(v_U) \setminus T_t(v)| \big] \le \frac{1}{p}.$$
It follows that, with positive probability, the random set $U$ satisfies 
$$|U| \le m, \qquad |T_{8\ell}(v)| \le 2 \cdot |F(v_U)| \qquad \text{and} \qquad F(v_U)  \subset T_t(v),$$ 
as required. 
\end{proof}

As observed above, it is now straightforward to complete the proof of Theorem~\ref{thm:containers}.

\begin{proof}[Proof of Theorem~\ref{thm:containers}]
Let $\cC$ be as defined in~\eqref{def:C}, and note that 
$$|\cC| \le p^m \le \exp\big( 2^{12} (\log p)^2 \big).$$ 
For each $v \in \Z_p^n$ with $\rho(v) \ge 4/p$ and $|v| \ge 2^{18} \log p$, let $Y$ and $U$ be the sets given by Lemmas~\ref{lem:Ydef} and~\ref{lemma:def:U} respectively, and define $B(v) := C\big( F(v_U) \big)$. 

Now, we have $n/4 \le |Y| \le n/2$, by Lemma~\ref{lem:Ydef}, and 
$$\big| \big\{ i \in [n] : v_i \not\in B(v) \big\} \big| \le \frac{n}{4},$$
by Lemma~\ref{lem:containers:contain:most:of:v}, since $F(v_U) \subset T_t(v)$, where $t = 2^{-7} n$, by Lemma~\ref{lemma:def:U}. Finally, we have 
$$|B(v)| \le \frac{4p}{|F(v_U)|} \le \frac{2^{15}}{\rho(v_Y) \sqrt{|v|}} \cdot \frac{|T_\ell(v_Y)|}{|F(v_U)|} \le \frac{2^{16}}{\rho(v_Y) \sqrt{|v|}},$$
by Lemmas~\ref{lem:containers:small:first:bound}--\ref{lemma:def:U}, since $|T_\ell(v_Y)| \le |T_{8\ell}(v)| \le 2 \cdot |F(v_U)|$. This completes the proof of the inverse Littlewood--Offord theorem. 
\end{proof}

\section{Applying the inverse Littlewood--Offord theorem}\label{sec:iterate}

In this section we will use our inverse Littlewood--Offord theorem to prove Lemma~\ref{thm:qrho}. Let us fix a sufficiently large integer $n \in \N$ and a prime $2 < p \le \exp\big( 2^{-10} \sqrt{n} \big)$ throughout the section. Recall that $\beta \ge 4/p$, that
$$q_n(\beta) = \max_{w \in \Z_p^n}\Pr\big( \exists \, v \in \Z_p^n \setminus \{0\} \,:\, M_n \cdot v = w \text{ and } \rho(v) \ge \beta \big),$$
and that our aim is to prove that $q_n(\beta) \le 2^{-n/4}$. We shall do so by using Theorem~\ref{thm:containers} to partition the vectors $v \in \Z_p^n$ with $\rho(v) \ge \beta$ and $|v| \ge 2^8 \sqrt{n}$ into a collection of `fibres', and then applying a simple first moment argument inside each fibre. Vectors with small support will require a separate (and much simpler) argument, so let us begin by dealing with those. For each $w \in \Z_p^n$, define 
$$Q(w) := \big| \big\{ v \in \Z_p^n \setminus \{0\} \,:\, M_n \cdot v = w \text{ and } |v| < 2^8 \sqrt{n} \big\} \big|.$$
Our first lemma bounds the expected size of $Q(w)$. 

\begin{lemma}\label{lem:small:support}
For every $w \in \Z_p^n$, 
$$\Ex\big[ Q(w) \big] \le 2^{-n/2}.$$
\end{lemma}

\begin{proof}
Fix $w \in \Z_p^n$; the lemma is an easy consequence of the following claim.   

\medskip
\noindent \textbf{Claim:} If $v \in \Z_p^n \setminus \{0\}$, then $\Pr\big( M_n \cdot v = w \big) \le 2^{-n}$.

\begin{proof}[Proof of Claim]
Choose $k \in [n]$ such that $v_k \neq 0$, and reveal the entire matrix $M_n$ except for the $k$th row and the $k$th column. Observe that if $M_n \cdot v = w$, then 
\begin{equation}\label{eq:roughprob}
m_{ik} v_k = w_i - \sum_{j \ne k} m_{ij} v_j
\end{equation}
for each $i \in [n]$, where $m_{ij}$ are the entries of $M_n$. Now, for any choice of the entries $m_{ij}$ with $j \ne k$, the event~\eqref{eq:roughprob} has probability at most $1/2$, and these events are independent for different values of $i \ne k$. Finally, having revealed the entire matrix except for $m_{kk}$, the event~\eqref{eq:roughprob} for $i = k$ has probability at most $1/2$, so $\Pr\big( M_n \cdot v = w \big) \le 2^{-n}$, as claimed.
\end{proof}

Now, since there are at most ${n \choose k} p^k$ vectors $v \in \Z_p^n \setminus \{0\}$ with $|v| < k$, and recalling that $p \le \exp\big( 2^{-10} \sqrt{n} \big)$ and $n$ is sufficiently large, the claim implies that 
$$\Ex\big[ Q(w) \big] \le {n \choose 2^8 \sqrt{n}} p^{2^8 \sqrt{n}} \cdot 2^{-n} \le 2^{-n/2}$$
as required. 
\end{proof}

From now on, we will therefore restrict our attention to the vectors with large support:
$$\cV := \big\{ v \in \Z_p^n :\, \rho(v) \ge \beta, \, |v| \ge 2^8\sqrt{n} \big\}.$$
To deal with these vectors, we will define a function 
$$f \colon \cV \rightarrow \cX := \Big\{ \big( X_i,Y_i,B_i \big)_{i=1}^\infty : X_i,Y_i \subset [n] \text{ and } B_i \subset \Z_p \text{ for each } i \in \N \Big\},$$ 
using Theorem~\ref{thm:containers}; the `fibres' forming our partition of $\cV$ will be exactly the fibres $f^{-1}(S)$ of the function $f$. We will define $f$ using the following algorithm, which takes as its input a vector $v \in \cV$, and outputs an element of $\cX$. 

\begin{alg}\label{alg:containers}
Let $v \in \mathcal{V}$. At the $k$th step, if the process has not yet ended, we will have constructed a sequence $(X_i,Y_i,B_i)_{i=1}^{k-1}$ with $X_i,Y_i \subset [n]$ and $B_i \subset \Z_p$ for each $i \in [k-1]$. In this case, set
$$Z_k := [n] \setminus \bigcup_{i=1}^{k-1} X_i,$$
and do the following:
\begin{itemize}
\item[1.] If $|v_{Z_k}| \ge 2^8\sqrt{n}$ then we apply Theorem~\ref{thm:containers}, and set $Y_k := Y(v_{Z_k})$, $B_k := B(v_{Z_k})$, and 
\begin{equation}\label{def:X:sets}
X_k := \big\{ i \in Z_k \setminus Y_k \,:\, v_i \in B_k \big\}.
\end{equation}
Set $k \to k + 1$ and repeat the process.
\item[2.] If $|v_{Z_k}| < 2^8\sqrt{n}$, then we set $k^* = k^*(v) := k - 1$ and 
$$X_j = Y_j = B_j = \emptyset$$
for every $j \ge k$. The process terminates, and we set $f(v) := \big( X_i,Y_i,B_i \big)_{i=1}^\infty$.
\end{itemize}
\end{alg}

Define $\mathcal{U} := \big\{ f(v) : v \in \mathcal{V} \big\}$. Theorem~\ref{thm:containers} implies the following upper bound on $|\cU|$.


\begin{lemma}\label{lem:Usize}
$$|\cU| \le n^{n/64}.$$
\end{lemma}

\begin{proof}
We claim first that, for each $k \in \N$, either $|v_{Z_k}| < 2^8\sqrt{n}$, or
\begin{equation}\label{eq:Zk:bound}
|Z_k| \le \left( \frac{3}{4} \right)^{k-1} n.
\end{equation}
Indeed, by Observation~\ref{obs:rho:monotone} we have $\rho(v_{Z_k}) \ge \rho(v) \ge \beta \ge 4/p$ for every $v \in \cV$, and therefore, if $|v_{Z_k}| \ge 2^8\sqrt{n} \ge 2^{18} \log p$, it follows from Theorem~\ref{thm:containers} that $|Y_k| \le |Z_k| / 2$ and 
\begin{equation}\label{eq:Xk:bound}
\big| Z_k \setminus \big( X_k \cup Y_k \big) \big| \le \big| \big\{ i \in Z_k : v_i \notin B_k \big\} \big| \le \frac{|Z_k|}{4}.
\end{equation}
Hence $|X_k| \ge |Z_k|/4$, and~\eqref{eq:Zk:bound} follows. In particular, this implies that $k^*(v) \le 2 \log n$.

Now, given $(X_i,Y_i,B_i)_{i=1}^{k-1}$, there are at most $2^{|Z_k|}$ choices for each of the sets $X_k$ and $Y_k$ (since they are subsets of $Z_k$), and by~\eqref{eq:number:of:containers} there are at most 
$$\exp\Big( 2^{12} \big( \log p \big)^2 \Big) \le \exp\big( 2^{-8} n \big)$$
choices for $B_k$. It follows that the total number of choices for $f(v)$ is at most
$$\exp\bigg( 2^{-7} n \log n + 2 \sum_{k = 1}^\infty \left( \frac{3}{4} \right)^{k-1} n \bigg) \le \exp\big( 2^{-6} n \log n \big) = n^{n/64},$$
as required, since $n$ is sufficiently large.
\end{proof}

We will bound, for each sequence $S \in \cU$, the probability that some vector $v \in \cV$ with $f(v) = S$ satisfies $M_n \cdot v = w$, and then sum over $S \in \cU$. To do so, for each $S \in \cU$ and $w \in \Z_p^n$, let us define a random variable
$$Q(S,w) := \big| \big\{ v \in \cV \,:\, f(v) = S \text{ and } M_n \cdot v = w \big\} \big|.$$
The next lemma bounds the expected size of $Q(S,w)$.

\begin{lemma}\label{lem:key:calc}
If $S = \big( X_i,Y_i,B_i \big)_{i=1}^\infty \in \cU$ and $w \in \Z_p^n$, then
\begin{equation}\label{eq:key:calc}
\Ex\big[ Q(S,w) \big] \, \le \, \bigg( \frac{2^{56}}{n} \bigg)^{n/16}.
\end{equation}
\end{lemma}

\begin{proof}
If $f(v) = S$, then we have $v_j \in B_i$ for every $j \in X_i$, and $|v_{Z_{k^*+1}}| < 2^8\sqrt{n}$. There are therefore at most
$${n \choose 2^8\sqrt{n}} \cdot p^{2^8\sqrt{n}} \cdot \prod_{i = 1}^{k^*} |B_i|^{|X_i|}$$
vectors $v \in \cV$ with $f(v) = S$. We claim that, for each such vector $v$, 
\begin{equation}\label{eq:prob:in:a:container}
\Pr\big( M_n \cdot v = w \big) \le \, \prod_{i=1}^{k^*} \max_{u(i) \in \Z_p^{|X_i|}} \Pr\big( M_{X_i \times Y_i}\cdot v_{Y_i} = u(i) \big) = \prod_{i=1}^{k^*} \rho(v_{Y_i})^{|X_i|}.
\end{equation}
To prove~\eqref{eq:prob:in:a:container}, recall from~\eqref{def:X:sets} that 
$$X_i \cap Y_i = \emptyset \qquad \text{and} \qquad X_i \cap X_j = Y_i \cap X_j = \emptyset$$ 
for every $i \in [k^*]$ and every $1 \le j < i$, since $X_i, Y_i \subset Z_i$. It follows that
$$\Pr\bigg( M_{X_i \times [n]} \cdot v = w_{X_i} \; \Big | \; \bigcap_{j = 1}^{i-1} M_{X_j \times [n]} \cdot v = w_{X_j} \bigg) \le \max_{u(i) \in \Z_p^{|X_i|}} \Pr\big( M_{X_i \times Y_i} \cdot v_{Y_i} = u(i) \big)$$ 
for every $i \in [k^*]$, and moreover the entries of $M_{X_i \times Y_i}$ are all independent. This proves~\eqref{eq:prob:in:a:container}, and summing over $v \in \cV$ with $f(v) = S$ gives
$$\Ex\big[ Q(S,w) \big] \, \le \, \max_{v \in \cV \,:\, f(v) = S} 
{n \choose 2^8\sqrt{n}} \cdot p^{2^8\sqrt{n}} \cdot \prod_{i = 1}^{k^*} \Big( |B_i| \cdot \rho(v_{Y_i}) \Big)^{|X_i|}.$$
To deduce~\eqref{eq:key:calc}, recall from Theorem~\ref{thm:containers} and Algorithm~\ref{alg:containers} that, for every $v \in \cV$ such that $f(v) = S$,
$$|B_i| \le \frac{2^{16}}{ \rho(v_{Y_i}) \sqrt{|v_{Z_i}|}} \le \frac{2^{12}}{ \rho(v_{Y_i}) n^{1/4}}$$
for each $i \in [k^*]$, since $|v_{Z_i}| \ge 2^8 \sqrt{n}$. Since $n \ge 2^{81}$ and $p \le \exp\big( 2^{-10} \sqrt{n} \big)$, and recalling from~\eqref{eq:Xk:bound} that we have $|X_1| \ge n/4$ (since $|v| \ge 2^8 \sqrt{n}$ for every $v \in \cV$), it follows that
$$\Ex\big[ Q(S,w) \big] \, \le \, {n \choose 2^8\sqrt{n}} \cdot p^{2^8\sqrt{n}} \cdot \bigg( \frac{2^{12}}{n^{1/4}} \bigg)^{\sum_i |X_i|} \le \bigg( \frac{2^{14}}{n^{1/4}} \bigg)^{n/4} = \bigg( \frac{2^{56}}{n} \bigg)^{n/16},$$
as required.
\end{proof}


Completing the proof of Lemma~\ref{thm:qrho}, and hence of Theorem~\ref{thm:main}, is now straightforward.

\begin{proof}[Proof of Lemma~\ref{thm:qrho}]
By Lemma \ref{lem:small:support}, for each $w\in \Z_p^n$ the probability that there exists $v \in \Z_p^n \setminus \{0\}$ such that $|v| < 2^8\sqrt{n}$ and $M_n \cdot v = w$ is at most $2^{-n/2}$, and hence
$$q_n(\beta) \le 2^{-n/2} + \sum_{S \in \cU} \max_{w\in \Z_p^n} \Pr\Big(\exists \, v \in \cV \,:\,  f(v) = S \text{ and } M_n \cdot v = w \Big).$$ 
Now, by Lemma~\ref{lem:key:calc}, we have 
$$\Pr\Big(\exists \, v \in \cV \,:\,  f(v) = S \text{ and } M_n \cdot v = w \Big) \, \le \, \bigg( \frac{2^{56}}{n} \bigg)^{n/16}$$
for every $S \in \cU$ and $w\in \Z_p^n$, and hence, by Lemma~\ref{lem:Usize},  
$$q_n(\beta) \le 2^{-n/2} + n^{n/64} \bigg( \frac{2^{56}}{n} \bigg)^{n/16} \le \, 2^{-n/4}$$
since $n$ is sufficiently large. This completes the proof of the lemma. 
\end{proof}

As observed in Section~\ref{sec:overview}, Lemmas~\ref{lem:sufftoshow} and~\ref{thm:qrho} together imply Theorem~\ref{thm:main}. 

{\setstretch{1.19}

\section*{Acknowledgement}

We would like to thank Asaf Ferber for a helpful conversation about the proof in~\cite{FeJ}, and the anonymous referees for their careful reading and many helpful suggestions.

}

\appendix

{\setstretch{1.2}

\section{Proof of Lemma~\ref{lem:sufftoshow}}\label{app:red}

In this appendix we will provide a proof of Lemma~\ref{lem:sufftoshow}, which allowed us to reduce the problem of bounding the probability that $\det(M_n) = 0$ to the problem of bounding $q_n(\beta)$. The proof given below is essentially contained in the paper of Ferber and Jain~\cite{FeJ}, and several of the key lemmas appeared in the papers of Costello, Tao and Vu~\cite{CTV} and Nguyen~\cite{N12}. However, for the benefit of the reader who is unfamiliar with the area, and since Lemma~\ref{lem:sufftoshow} was not stated explicitly in~\cite{FeJ}, we will provide the details in full. 
We begin by giving an overview of the proof.

\subsection{Overview of the proof of Lemma~\ref{lem:sufftoshow}}\label{app:overview:sec}

It will be convenient in this section to work over $\F_p$; in particular, we will consider the entries of $M_n$ as elements of $\F_p$, noting that doing so can only increase the probability that $M_n$ is singular. Observe also that 
\begin{equation}\label{app:qdef}
q_n(\beta) = \max_{w \in \F_p^n} \Pr\Big( \exists \, v \in \F_p^n \setminus \{0\}: M_n \cdot v = w \text{ and } \rho(v) \ge \beta \Big),
\end{equation}
where now $M_n$ is a matrix over $\F_p$. 

Let us write $\rank(M)$ for the rank of a matrix $M$ over $\F_p$, and $M_{n-1}$ for the random symmetric matrix obtained by removing the first row and column from $M_n$. The following lemma, which was proved by Nguyen (see~\cite[Section~2]{N12}), allows us to restrict our attention to matrices $M_n$ such that $\rank(M_{n}) = n-1$ and $\rank(M_{n-1}) \in \{n-2,n-1\}$. 

\begin{lemma}\label{lem:red1}
For every $n \in \N$ and prime $p > 2$,
$$\Pr\big( \det(M_n) = 0 \big) \le 4n \sum_{m = n}^{2n-2} \Pr\Big( \big\{ \rank(M_m) = m - 1 \big\} \cap \big\{ \rank(M_{m-1}) \in \{m-2,m-1\} \big\} \Big).$$
\end{lemma}

The proof of Lemma~\ref{lem:red1} is given in Section~\ref{subsec:reductions}. The next two lemmas deal with the cases $\rank(M_{n-1}) = n - 2$ and $\rank(M_{n-1}) = n - 1$ respectively; the first is more straightforward. 

\begin{lemma}\label{lem:rank2}
For every $n \in \N$, prime $p > 2$, and $\beta > 0$, 
$$\Pr\Big( \big\{ \rank(M_{n}) = n-1 \big\} \cap \big\{ \rank(M_{n-1}) = n-2 \big\} \Big) \le \beta + q_{n-1}(\beta).$$
\end{lemma}

The proof of Lemma~\ref{lem:rank2}, which follows that given in~\cite[Section 2.2]{FeJ}, is described in Section~\ref{subsec:n2}. Finally, the following lemma deals with the case $\rank(M_{n-1}) = n-1$.

\begin{lemma}\label{lem:rank1}
For every $n \in \N$, prime $p > 2$, $\beta  > 0$, and integer $1 \le k \le n - 2$, we have
$$\Pr\Big( \rank(M_{n}) = \rank(M_{n-1}) = n-1 \Big) \le \, 2 \cdot \big( 2^k \beta + 2^{-k} \big)^{1/4} + 3^{k+1} q_{n-1}(\beta).$$
\end{lemma}

The proof of Lemma~\ref{lem:rank1}, which is similar to that given in~\cite[Section~2.3]{FeJ}, is provided in Section~\ref{subsec:n1}. Combining Lemmas~\ref{lem:red1},~\ref{lem:rank2} and \ref{lem:rank1}, we obtain Lemma~\ref{lem:sufftoshow}.

}

\medskip
\pagebreak

{\setstretch{1.15}

\begin{proof}[Proof of Lemma~\ref{lem:sufftoshow}]
Observe first that $q_n(\beta) \ge 2^{-n}$ for every $\beta < 1/2$ (to see this, set $v = (1,0,\ldots,0)$), so the claimed bound holds trivially if $\beta > n^{-1}$ or $\beta < 2^{-n}$. We may therefore assume that $k := \lfloor \log_4(1/\beta) \rfloor \in [n-2]$, and therefore, by Lemmas~\ref{lem:red1},~\ref{lem:rank2} and~\ref{lem:rank1}, we obtain
\begin{align*}
\Pr\big( \det(A_n) = 0 \big) & \, \le \, 4n \sum_{m = n}^{2n-2} \Big( \beta + q_{m-1}(\beta) + 2 \cdot \big( 3 \beta^{1/2} \big)^{1/4} + \beta^{-1} q_{m-1}(\beta) \Big)\\
& \, \le \, 16n \sum_{m = n-1}^{2n-3} \bigg( \beta^{1/8} + \frac{q_m(\beta)}{\beta} \bigg).
\end{align*}
as required.
\end{proof}

\subsection{The proof of Lemma~\ref{lem:red1}}\label{subsec:reductions}

As noted above, Lemma~\ref{lem:red1} is a straightforward consequence of~\cite[Lemmas~2.1 and~2.3]{N12}. The first of these two lemmas is as follows.

\begin{lemma}[Lemma~2.1 of~\cite{N12}]\label{lem:Ngu:10}
For any $0 \le k \le n - 1$,
\begin{equation}\label{eq:rank}
\Pr\big( \rank(M_n) = k \big) \le 2 \cdot \Pr\big( \rank(M_{2n-k-1}) = 2n-k-2 \big).
\end{equation}
\end{lemma}


To prove Lemma~\ref{lem:Ngu:10} we will need the following observation of Odlyzko~\cite{Odl}; since it is usually stated in $\R^n$, we provide the short proof.

\begin{obs}\label{obs:odl}
Let $V$ be a subspace of $\F_p^n$ of dimension at most $k$. Then 
$$\big| V \cap \{ -1, 1 \}^n \big| \le 2^k.$$
\end{obs}

\begin{proof}
Form an $n \times k$ matrix over $\F_p$ whose columns are a basis $\{v^{(1)},\ldots,v^{(k)}\}$ of $V$, and choose $k$ linearly independent rows. We obtain an invertible matrix $A$, and so for each $b \in \{-1,1\}^k$, there is a unique solution in $\F_p^k$ to the set of equations $A x = b$. The $2^k$ vectors $\sum_{i=1}^k x_i v^{(i)}$ (one for each $b \in \{-1,1\}^k$) are the only possible elements of $V \cap \{ -1, 1 \}^n$.
\end{proof}
 
We can now prove Lemma~\ref{lem:Ngu:10}.

\begin{proof}[Proof of Lemma~\ref{lem:Ngu:10}]
We claim that, for any $0 \le k \le n - 1$, 
\begin{equation}\label{eq:rankit}
\Pr\big( \rank(M_{n+1}) = k + 2 \;\big|\; \rank(M_n) = k  \big) \ge 1 - 2^{k - n}.
\end{equation}
where we remind the reader that $M_n$ is obtained from $M_{n+1}$ by removing the first row and column. Let $W$ be the subspace spanned by the rows of $M_n$, and note that, by Observation~\ref{obs:odl}, if $\rank(M_n) = k$ then $W$ intersects $\{-1,1\}^n$ in at most $2^{k}$ vectors. 

Let $v \in \F_p^n$ be the vector formed by removing the first element from the first row of $M_{n+1}$. By the remarks above, it follows that $\Pr(v \notin W) \ge 1 - 2^{k-n}$. We claim that if $v \not\in W$ then $\rank(M_{n+1}) = k + 2$. To see this, note first that if $v \not\in W$ then the rank of the final $n$ columns of $M_{n+1}$ is $k+1$. Now, since $M_{n+1}$ is symmetric, the first column of $M_{n+1}$ is the same as the first row, and if $v \not\in W$ then $v$ is not in the span of the columns of $M_n$. It follows that $\rank(M_{n+1}) = k+2$, as claimed, and~\eqref{eq:rankit} follows.

It follows immediately from~\eqref{eq:rankit} that  
$$\Pr\big( \rank(M_{n+t}) = k + 2t \;\big|\; \rank(M_{n+ t-1}) = k + 2t-2) \big) \ge 1 - 2^{k + t - n - 1}$$
for every $k \ge 0$ and $1 \le t \le n-k$. Now, building $M_{n+t}$ from $M_n$ by adding one row and column at a time, it follows that
$$\Pr\big( \rank(M_{2n-k-1}) = 2n - k -2 \;\big|\; \rank(M_n) = k \big) \ge  \prod_{i = 2}^{n-k}(1 - 2^{-i}) \ge \frac{1}{2},$$
which implies~\eqref{eq:rank}.
\end{proof}

We can now deduce Lemma~\ref{lem:red1} using~\cite[Lemma~2.3]{N12}, which is the following observation. Let us write $M_n^{(i)}$ for the (symmetric) matrix obtained from $M_n$ by removing the $i$th row and the $i$th column. 

\begin{lemma}[Lemma~2.3 of~\cite{N12}]\label{lem:Ngu:remove}
If\/ $\rank(M_n) = n-1$, then\/ $\max_{i \in [n]} \rank(M_n^{(i)}) \ge n - 2$.
\end{lemma}

\begin{proof}
Choose $n-1$ rows of $M_n$ whose span has dimension $n-1$, and remove the remaining row, giving an $(n-1) \times n$ matrix of rank $n-1$. Hence, removing any column from this matrix, we obtain a matrix of rank at least $n-2$. 
\end{proof}

\begin{proof}[Proof of Lemma~\ref{lem:red1}]
By Lemma~\ref{lem:Ngu:10}, we have
\begin{equation}\label{eq:rkbound}
\Pr\big( \det(M_n) = 0 \big) = \, \sum_{k = 1}^{n-1}\Pr\big( \rank(M_n) = k \big) \le \, 2 \sum_{k = 1}^{n-1} \Pr\big( \rank(M_{2n - k - 1}) = 2n - k - 2 \big).
\end{equation}
We therefore need to bound $\Pr\big( \rank(M_m) = m - 1 \big)$ for each $n \le m \le 2n - 2$. By Lemma~\ref{lem:Ngu:remove}, and by symmetry, we have
\begin{align*}
\Pr\big( \rank(M_m) = m-1) & \, \le \, \sum_{i = 1}^m \Pr\Big( \big\{ \rank(M_m) = m - 1 \big\} \cap \big\{ \rank(M_m^{(i)}) \ge m - 2 \big\} \Big) \\
& \, \le \, m \cdot \Pr\Big( \big\{ \rank(M_m) = m - 1 \big\} \cap \big\{ \rank(M_{m-1}) \in \{m - 2,m-1\} \big\} \Big).
\end{align*}
Combining this with \eqref{eq:rkbound} gives the statement of the lemma.
\end{proof}

\subsection{The case $\rank(M_{n-1}) = n-2$}\label{subsec:n2}

In this subsection we will prove Lemma~\ref{lem:rank2}, following the presentation in~\cite[Section~2.2]{FeJ}. Let us write $\adj(M)$ for the \emph{adjugate} of a matrix $M$ over $\F_p$. 
We will need the following lemma of Nguyen~\cite{N12}, see~\cite[Lemma~2.5]{FeJ}. 

\begin{lemma}\label{lem:ngured}
If $\rank(M_{n-1}) = n-2$, then there exists a non-trivial column $a \in \F_p^{n-1}$ of $\adj(M_{n-1})$ such that
\begin{itemize}
\item[$(a)$] $M_{n-1} \cdot a = 0$, and 
\item[$(b)$] if $\det(M_n) = 0$, then $\sum_{i=2}^{n} a_i x_i = 0$,
\end{itemize}
where $a = (a_2,\ldots,a_n)$, and $(x_1,\ldots, x_n)$ is the first row of $M_n$. 
\end{lemma}

\begin{proof}
Recall (see, e.g.,~\cite[page~22]{HJ}) that if $\rank(M_{n-1}) = n - 2$, then 
$$M_{n-1} \cdot \adj(M_{n-1}) = 0 \qquad \text{and} \qquad \rank\big( \adj(M_{n-1}) \big) = 1.$$
It follows that there exists a non-trivial column vector $a$ of $\adj(M_{n-1})$, and $M_{n-1} \cdot a = 0$.

To show that property $(b)$ holds, recall that, since $M_n$ is symmetric, 
$$\det(M_n) = x_1 \det(M_{n-1}) - \sum_{2 \le i,j \le n} c_{ij} x_i x_j,$$ 
where $c_{ij}$ are the entries of $\adj(M_{n-1})$. Since $\adj(M_{n-1})$ is a symmetric matrix of rank~$1$, its entries can be written in the form $c_{ij} = \lambda a_i a_j$ for some $\lambda \in \F_p \setminus \{ 0\}$. Hence
\begin{equation}\label{eq:factor}
0 = \sum_{2 \le i,j \le n}a_i a_j x_i x_j = \bigg( \sum_{2 \le i\le n}a_i x_i \bigg)^2,
\end{equation}
since $\det(M_{n-1}) = \det(M_n) = 0$, as required.
\end{proof}

We now use Lemma~\ref{lem:ngured} to deduce Lemma~\ref{lem:rank2}, cf.~\cite[Section~2.2]{FeJ}.

\begin{proof}[Proof of Lemma~\ref{lem:rank2}]
By Lemma~\ref{lem:ngured}, it follows that in order to bound the probability that $\rank(M_n) = n-1$ and $\rank(M_{n-1}) = n-2$, it suffices to bound the probability that there exists a vector $a \in \F_p^{n-1} \setminus \{0\}$ (unique up to a constant factor) with $M_{n-1} \cdot a = 0$ and $a \cdot x = 0$, where $x \in \{-1,1\}^{n-1}$ is a random vector chosen uniformly and independent of $M_{n-1}$.

We will partition this event into `structured' and `unstructured' cases, using the event
$$\mathcal{U}_{\beta} := \big\{ \rho(v) \le \beta \text{ for every vector $v \in \F_p^{n-1} \setminus \{0\}$ with $M_{n-1}\cdot v = 0$} \big\}.$$
Observe first that, for any $M_{n-1} \in \cU_\beta$, and any $a \in \F_p^{n-1} \setminus \{0\}$ with $M_{n-1} \cdot a = 0$, we have
$$\Pr\big( a \cdot x = 0 \,\big|\, M_{n-1} \big) \le \beta,$$
and hence
$$\Pr\Big( \big\{ \rank(M_{n}) = n-1 \big\} \cap \big\{ \rank(M_{n-1}) = n-2 \big\} \cap \cU_\beta \Big) \le \beta.$$
On the other hand, by the definition~\eqref{app:qdef} of $q_n(\beta)$, we have
$$\Pr\big( \cU^c_\beta \big) = \Pr\big( \exists \, v \in \F_p^{n-1}\setminus \{0\} : M_{n-1} \cdot v = 0 \text{ and } \rho(v) > \beta \big) \le q_{n-1}(\beta).$$
It follows that
$$\Pr\Big( \big\{ \rank(M_{n}) = n-1 \big\} \cap \big\{ \rank(M_{n-1}) = n-2 \big\} \Big) \le \beta + q_{n-1}(\beta),$$
as required.
\end{proof}

\subsection{The case $\rank(M_{n-1}) = n-1$}\label{subsec:n1}

It only remains to prove Lemma~\ref{lem:rank1}. The strategy is similar to that used in the previous subsection (in particular, we will split our event into a `structured' case and an `unstructured' case), but now it is trickier to relate our event to $q_n(\beta)$, as we do not have the simple factorisation of the determinant used in Lemma~\ref{lem:ngured}. Instead, we will apply the following `decoupling' lemma of Costello, Tao and Vu~\cite{CTV}. 

\begin{lemma}[Lemma~4.7 of~\cite{CTV}]\label{lem:decouple}
Let $X$ and $Y$ be independent random variables, and let $\cE(X,Y)$ be an event that depends on $X$ and $Y$. Then
$$\Pr\big( \cE(X,Y) \big) \le \Big( \Pr\big( \cE(X,Y) \cap \cE(X',Y) \cap \cE(X,Y') \cap \cE(X',Y') \big) \Big)^{1/4},$$
where $X'$ and $Y'$ are independent copies of $X$ and $Y$.
\end{lemma}

It was remarked in~\cite{CTV} that Lemma~\ref{lem:decouple} is equivalent to the classical fact (which was essentially proved by Erd\H{o}s~\cite{E38} in 1938) that a bipartite graph with parts of size $m$ and $n$ and $cmn$ edges contains at least $c^4 m^2 n^2$ (possibly degenerate) copies of $C_4$. Indeed, to deduce Lemma~\ref{lem:decouple} from this theorem, simply define a bipartite graph, each of whose vertices represents an element of the range of $X$ or $Y$, and whose edges encode the event $\cE$.

In order to state the key technical lemma that we will use to prove Lemma~\ref{lem:rank1}, we need a little notation. Given a vector $v \in \F_p^m$ and a set $J \subset [m]$, let $v_J \in \F_p^{|J|}$ denote the restriction of $v$ to the coordinates of $J$, and let $v^*_J$ be the vector in $\F_p^m$ whose $i$th coordinate is $v_i \cdot \mathbbm{1}[i \in J]$. Moreover, let $u,u' \in \{-1,1\}^{n-1}$ be chosen uniformly and independently at random, and define $w \in \{ -2, 0, 2\}^{n-1}$ by setting $w_i := u_i - u'_i$ for each $i \in [n-1]$.

The following lemma was essentially proved in~\cite[Section~4.6]{CTV} (see also~\cite[Section~2.3]{FeJ}).

\begin{lemma}\label{lem:dec}
For any non-trivial partition $I \cup J = [n - 1]$, we have
$$\Pr\big( \rank(M_n) = \rank(M_{n-1}) = n-1 \big) \le 2 \cdot \Ex\Big[ \max_{a \in \F_p} \Pr\big( z_I \cdot w_I = a \,\big|\, M_{n-1} \big)^{1/4} \mathbbm{1}\big[ \rank(M_{n-1}) = n-1 \big] \Big],$$
where $z := M_{n-1}^{-1} \cdot w^*_J$, and the expectation is over the choice of $M_{n-1}$. 
\end{lemma}

\begin{proof}
{\setstretch{1.13}
Let $X := (u_i)_{i \in I}$ and $Y := (u_i)_{i \in J}$ be random variables, and note that $u$ is determined by $X$ and $Y$. Now define, for each choice of $M_{n-1}$, an event 
$$\cE(X,Y) := \big\{ \exists \, v \in \F_p^{n-1} : M_{n-1} \cdot v = u \text{ and } u \cdot v \in \{-1,1\} \big\}$$
depending on $X$ and $Y$. We claim that if $\rank(M_{n-1}) = n-1$, and the first row of $M_n$ is $(x_1,u_1,\ldots,u_{n-1})$ for some $x_1 \in \{-1,1\}$, then
$$\big\{ \rank(M_n) = n-1 \big\} \Rightarrow \big\{ u \in \cE(X,Y) \big\}.$$
Indeed, since $\det(M_n) = 0 \ne \det(M_{n-1})$ there exists a vector $v \in \F_p^n$ such that $M_n \cdot v = 0$ and $v_1 = -1$. Letting $v' = (v_2,\ldots,v_n)$, we see that $M_{n-1} \cdot v' = u$ and $u \cdot v' \in \{-1,1\}$. 

Now, for each choice of $M_{n-1}$, define\footnote{Here we set $X' := (u'_i)_{i \in I}$ and $Y' := (u'_i)_{i \in J}$, so $X'$ and $Y'$ are independent copies of $X$ and $Y$.} 
$$\cE_1 := \cE(X,Y) \cap \cE(X',Y) \cap \cE(X,Y') \cap \cE(X',Y').$$
By Lemma~\ref{lem:decouple}, we have
$$\Pr\big( \cE(X,Y) \,\big|\, M_{n-1} \big) \le \Pr\big( \cE_1 \,\big|\, M_{n-1} \big)^{1/4},$$
and hence
\begin{align*}\label{eq:lemap}
\Pr\Big( \rank(M_n) = \rank(M_{n-1}) = n-1 \Big) & \, \le \, \Ex\Big[ \Pr\big( \cE(X,Y) \,\big|\, M_{n-1} \big) \mathbbm{1}\big[ \rank(M_{n-1}) = n-1 \big] \Big] \\
& \, \le \, \Ex\Big[ \Pr\big( \cE_1 \,\big|\, M_{n-1} \big)^{1/4} \mathbbm{1}\big[ \rank(M_{n-1}) = n-1 \big] \Big],
\end{align*}
where the expectation is over the choice of $M_{n-1}$. 

To complete the proof of the lemma, it will therefore suffice to show that
\begin{equation}\label{eq:prob:E1:16}
\Pr\big( \cE_1 \,\big|\, M_{n-1} \big) \le 16 \cdot \max_{a \in \F_p} \Pr\big( z_I \cdot w_I = a \,\big|\, M_{n-1} \big)
\end{equation}
for all $M_{n-1}$ with $\rank(M_{n-1}) = n-1$. To prove~\eqref{eq:prob:E1:16}, let us fix $M_{n-1}$ (arbitrarily among those with $\rank(M_{n-1}) = n-1$) and set $A := M_{n-1}^{-1}$ and $D := \{-1, 1\}$. We claim that if $u \in \cE(X,Y)$, then $u^T A u \in D$. To see this, simply observe that 
$$u^T A u = u^T A \cdot M_{n-1} v = u^T v \in \{-1, 1\} = D.$$ 
Recalling that $u = u(X,Y)$, define $f(X,Y) := u^T A u$, and observe that if $\cE_1$ holds, then 
\begin{equation}\label{eq:ustar}
f(X,Y) - f(X',Y) - f(X,Y') + f(X',Y') \in 2D - 2D,
\end{equation}
by the observation above. We claim that the left-hand side of~\eqref{eq:ustar} is equal to $2 z_I \cdot w_I$. To see this, note that 
$$f(X,Y) = u^T A u = \sum_{1 \le i,j \le n-1}A_{ij} u_iu_j,$$
}and (abusing notation) let us write $f(X,Y)_{ij} := A_{ij} u_iu_j$. Now, observe that if $i,j \in I$, then $f(X,Y)_{ij} = f(X,Y')_{ij}$ and $f(X',Y)_{ij} = f(X',Y')_{ij}$, and therefore 
$$\sum_{i,j \in I} f(X,Y)_{ij} - f(X',Y)_{ij} - f(X,Y')_{ij} + f(X',Y)_{ij} = 0.$$
Similarly, if $i,j \in J$ then $f(X,Y)_{ij} = f(X',Y)_{ij}$ and $f(X,Y')_{ij} = f(X',Y')_{ij}$, and hence
$$f(X,Y) - f(X',Y) - f(X,Y') + f(X',Y') = 2 \sum_{i \in I}\sum_{ j \in J}A_{ij}(u_i - u_i')(u_j - u_j').$$
Recalling that $w = u - u'$ and $z_i := \sum_{j \in J} A_{ij} w_j$, it follows that
\begin{equation*}\label{eq:zw}
    z_I \cdot w_I = \sum_{i \in I} (u_i - u_i') \sum_{j \in J} A_{ij}(u_j - u_j'),
\end{equation*}
so the left-hand side of~\eqref{eq:ustar} is equal to $2 z_I \cdot w_I$, as claimed. Since $|D| = 2$, it follows that
$$\Pr\big( \cE_1 \,\big|\, M_{n-1} \big) \le 16 \cdot \max_{a \in \F_p} \Pr\big( z_I \cdot w_I = a \,\big|\, M_{n-1} \big),$$
as claimed. As noted above, this completes the proof of the lemma.
\end{proof}

In the proof of Lemma~\ref{lem:rank1} we will need the following variant of $\rho(v)$. For any $n \in \N$ and $v \in \F_p^n$, define
$$\rho_{1/2}(v):= \max_{a \in \F_p} \Pr\big( u_1 v_1 + \cdots + u_n v_n = a \big),$$
where $u_1,\ldots,u_n$ are iid random variables taking the value $0$ with probability $1/2$, and the values $\pm 1$ each with probability $1/4$. We will need the following simple inequalities. 

\begin{lemma}[Lemma~2.8 and~2.9 of~\cite{FeJ}]\label{lem:rhobound}
For any $v \in \F_p^n$, and any partition $I \cup J = [n]$,
$$\rho_{1/2}(v) \le \rho(v) \qquad \text{and} \qquad \rho(v) \le \rho(v_I) \le 2^{|J|} \rho(v).$$
\end{lemma}

\begin{proof}
Observe first that 
$$\rho(v) \le \sum_{w \in \{ -1, 1 \}^{|J|}} \Pr\big( u_J = w \big) \cdot \max_{a \in \F_p} \Pr\big( u_I \cdot v_I = a - w \cdot v_J \,\big|\, u_J = w \big) \, \le \, \rho (v_I).$$
Since $\rho_{1/2}(v) = \rho(v \oplus v)$, it follows that $\rho_{1/2}(v) \le \rho(v)$. Finally, if $a \in \F_p$ maximises $\Pr\big( u_I \cdot v_I = a \big)$, then 
$$\rho(v_I) = 2^{|J|} \cdot \Pr\big( u_I \cdot v_I = a \big) \prod_{j \in J} \Pr\big( u_j =1 \big) \le 2^{|J|} \cdot \Pr\bigg( u \cdot v = a + \sum \limits_{j \in J} v_j \bigg) \le 2^{|J|} \cdot \rho(v),$$
as claimed.
\end{proof}

We are now ready to prove our final lemma, cf.~\cite[Section~2.3]{FeJ}.

\begin{proof}[Proof of Lemma~\ref{lem:rank1}]
Recall that $1 \le k \le n - 2$, and let $J \subset [n-1]$ with $|J| = k$. By Lemma~\ref{lem:dec}, it will suffice to show that
$$\Ex\Big[ \max_{a \in \F_p} \Pr\big( z_I \cdot w_I = a \,\big|\, M_{n-1} \big)^{1/4} \mathbbm{1}\big[ \rank(M_{n-1}) = n-1 \big] \Big] \le \big( 2^{|J|} \beta + 2^{-|J|} \big)^{1/4} + 3^{|J|} q_{n-1}(\beta),$$
where $I = [n] \setminus J$ and $z = M_{n-1}^{-1} \cdot w^*_J$ is defined whenever $\rank(M_{n-1}) = n-1$. Recall that $w \in \{ -2, 0, 2\}^{n-1}$, and observe that therefore $M_{n-1} \cdot z = w_J^* \in W(J)$, where 
$$W(J) := \big\{ v \in \{-2,0,2\}^{n-1} : v_j = 0 \text{ for all } j \not\in J \big\}.$$
We will use the following event to partition into cases:
$$\cU_\beta^{(J)} := \Big\{ \rho(v) \le \beta \text{ for every vector $v \in \F_p^{n-1} \setminus \{0\}$ such that $M_{n-1}\cdot v \in W(J)$} \Big\}.$$ 
We will bound the expectation above using the following three claims. 

\bigskip
\noindent \textbf{Claim 1:} $\Pr\big( M_{n-1} \not\in \cU_\beta^{(J)} \big) \le 3^{|J|} q_{n-1}(\beta)$.

\begin{proof}[Proof of Claim~1]
If $\cU_\beta^{(J)}$ does not hold for $M_{n-1}$, then there exists a vector $v \in \F_p^{n-1} \setminus \{0\}$ such that $M_{n-1}\cdot v \in W(J)$ and $\rho(v) 
 > \beta$. 
For each individual vector $w \in W(J)$, the probability that this holds with $M_{n-1}\cdot v = w$ is at most $q_{n-1}(\beta)$, by~\eqref{app:qdef}. Hence, summing over $w \in W(J)$, and noting that $|W(J)| = 3^{|J|}$, the claim follows.
\end{proof}

\medskip
\noindent \textbf{Claim 2:} If $\rank(M_{n-1}) = n-1$, then $\Pr\big( z = 0 \,\big|\, M_{n-1} \big) \le 2^{-|J|}$.

\begin{proof}[Proof of Claim~2]
If $z = 0$ then $w_J^* = M_{n-1} \cdot z = 0$. Since $w_i = 0$ occurs with probability $1/2$ for each $i \in J$, and these events are independent, the claim follows immediately. 
\end{proof}

\medskip
\noindent \textbf{Claim 3:} If $M_{n-1} \in \cU_\beta^{(J)}$ and $\rank(M_{n-1}) = n-1$, then 
$$\ds\max_{a \in \F_p} \Pr\Big( \big\{ z_I \cdot w_I = a \big\} \cap \big\{ z \ne 0 \big\} \;\big|\; M_{n-1} \Big) \le 2^{|J|} \beta.$$

\begin{proof}[Proof of Claim~3]
Recall that $w_J$ and $M_{n-1}$ together determine $z$, and that the entries of $w_I$ are independent of $w_J$, and observe that $\rho_{1/2}(z_I) = \max_{a \in \F_p} \Pr\big( z_I \cdot w_I = a \big)$. Therefore
$$\Pr\Big( \big\{ z_I \cdot w_I = a \big\} \cap \big\{ z \ne 0 \big\} \,\big|\, M_{n-1} \Big) \le \Ex\Big[ \rho_{1/2}(z_I) \mathbbm{1}\big[ z \ne 0 \big] \;\big|\; M_{n-1} \Big]$$
for every $a \in \F_p$, where the expectation is over the choice of $w_J$. Now, by Lemma~\ref{lem:rhobound},
$$\rho_{1/2}(z_I) \le \rho(z_I) \le 2^{|J|} \rho(z).$$
Since $M_{n-1} \in \cU_\beta^{(J)}$ and $M_{n-1} \cdot z = w_J^* \in W(J)$, if $z \ne 0$ then $\rho(z) \le \beta$. It follows that
$$\Ex\Big[ \rho_{1/2}(z_I) \mathbbm{1}\big[ z \ne 0 \big] \;\big|\; M_{n-1} \Big] \le 2^{|J|} \beta,$$
as claimed.
\end{proof}

By Claims~1,~2 and~3, it follows that 
$$\Ex\Big[ \max_{a \in \F_p} \Pr\big( z_I \cdot w_I = a \,\big|\, M_{n-1} \big)^{1/4} \mathbbm{1}\big[ \rank(M_{n-1}) = n-1 \big] \Big] \le \big( 2^{|J|} \beta + 2^{-|J|} \big)^{1/4} + 3^{|J|} q_{n-1}(\beta),$$
and, as noted above, this completes the proof of Lemma~\ref{lem:rank1}.
\end{proof}

As shown in Section~\ref{app:overview:sec}, this completes the proof of Lemma~\ref{lem:sufftoshow}.

\section{Hal\'{a}sz's Anticoncentration Lemma}\label{app:Halasz}


In this appendix we will provide, for completeness, a proof of Lemma~\ref{lem:halasz}, which is due to Hal\'asz~\cite{H77}. Let us fix a prime $p > 3$ and an integer $n \in \N$; the first step is the following bound on $\rho(v)$. Recall that $\| \cdot \|$ denotes the distance to the nearest integer.

\begin{lemma}\label{lem:H:first:bound}
For every $v \in \Z_p^n$, 
\begin{equation}\label{eq:H:first:bound}
\rho (v) \le \dfrac{1}{p} \cdot \sum \limits_{k \in \Z_p} \exp \left (-  \sum \limits_{j = 1}^{n}  \left \|\dfrac{k \cdot v_j}{p} \right \|^2 \right ).
\end{equation}
\end{lemma}

\begin{proof}
We need to bound, for each $a \in \Z_p$, the probability that $u \cdot v = a$, where $u$ is chosen uniformly at random from $\{-1,1\}^n$. The first step is to rewrite this probability as
$$\Pr\big( u \cdot v = a \big) = \dfrac{1}{p} \cdot \sum \limits_{k \in \Z_p} \Ex\bigg[ \exp\bigg( \frac{2 \pi i \cdot (u \cdot v - a) k}{p} \bigg) \bigg],$$
using the fact that $\sum_{k \in \Z_p} \exp\big( 2\pi i \cdot x k / p \big) = 0$ for every $x \in \Z_p \setminus \{0\}$. Now, noting that 
$$\Ex\bigg[ \exp\bigg( \frac{2 \pi i \cdot u_j v_j k}{p} \bigg) \bigg] = \dfrac{1}{2} \Big( e^{2 \pi i k v_j / p} + e^{- 2 \pi i k v_j / p} \Big) = \cos\bigg( \frac{2 \pi k \cdot v_j}{p} \bigg)$$
for each $k \in \Z_p$ and $j \in [n]$, and recalling that the $u_j$ are independent, it follows that
\begin{align*}
\Pr\big( u \cdot v = a \big) & \, = \, \dfrac{1}{p} \cdot \sum \limits_{k \in \Z_p} \exp\bigg( - \frac{2 \pi i \cdot a \cdot k}{p} \bigg) \prod_{j = 1}^{n} \cos\bigg( \frac{2 \pi k \cdot v_j}{p} \bigg)\\
& \, \le \, \dfrac{1}{p} \cdot \sum \limits_{k \in \Z_p} \prod_{j = 1}^{n} \bigg| \cos\bigg( \frac{\pi k \cdot v_j}{p} \bigg) \bigg|,
\end{align*}
where we used the fact that $\{ 2k : k \in \Z_p \} = \Z_p$. 

Finally, using the inequality $\big |\cos(\pi x/p) \big | \le \exp\big( - \| x/p \|^2 \big)$, we obtain
$$\rho(v) = \max_{a \in \Z_p} \Pr\big( u \cdot v = a \big) \le \dfrac{1}{p} \cdot \sum \limits_{k \in \Z_p} \exp\bigg( - \sum_{j = 1}^n \bigg\|\dfrac{k \cdot v_j}{p} \bigg\|^2 \bigg),$$
as claimed. 
\end{proof}

We next rewrite the right-hand side of~\eqref{eq:H:first:bound} in terms of the level sets $T_t(v)$. 

\begin{lemma}\label{lem:H:second:bound}
For every $v \in \Z_p^n \setminus \{0\}$ and $\ell \ge 1$, 
\begin{equation}\label{eq:H:second:bound}
\rho (v) \le \dfrac{1}{p} + \dfrac{e}{p} \sum_{t =1}^{\lceil \ell \rceil} e^{-t} \big| T_t(v) \big| + e^{-\ell}.
\end{equation}
\end{lemma}

\begin{proof}
By Lemma~\ref{lem:H:first:bound} and the definition~\eqref{def:level:sets} of $T_t(v)$, we have
$$\rho(v)\le \dfrac{1}{p} \bigg( |T_0(v)| + \sum \limits_{t=1}^{n} \big | T_t(v) \setminus T_{t-1}(v) \big| \cdot e^{-(t-1)} \bigg).$$
Now observe that $T_0(v) =\{0\}$, since $v \ne 0$, and therefore
$$\rho(v) \, \le \, \dfrac{1}{p} + \dfrac{e}{p} \sum_{t =1}^{\lceil \ell \rceil} e^{-t} \big| T_t(v) \big| + e^{-\ell}$$
for any $\ell \ge 1$, as required. 
\end{proof}

In order to deduce Lemma~\ref{lem:halasz} from Lemma~\ref{lem:H:second:bound}, we will need the following simple lemma. 

\begin{lemma} \label{lem:msumsetofT}
For any $m \in \N$ and $t \ge 0$, and any vector $v \in \Z_p^n$, 
$$m \cdot T_t(v) \subset T_{m^2t}(v)$$
where $m \cdot T$ denotes the $m$-fold sumset of a set $T$.
\end{lemma}

\begin{proof}
For each $a_1,\ldots a_m \in T_t(v)$, we have
$$\sum_{k=1}^n \bigg\| \sum_{j=1}^{m} \frac{a_j \cdot v_k}{p} \bigg\|^2 \le \, \sum_{k=1}^n \bigg( \sum_{j=1}^{m} \bigg\| \frac{a_j \cdot v_k}{p} \bigg\| \bigg)^2 \le m \sum_{j=1}^m \sum_{k=1}^n \bigg\| \frac{a_j \cdot v_k}{p} \bigg\|^2 \le m^2 t$$
by the triangle inequality for $\|\cdot\|$, convexity, and the definition of $T_t(v)$. 
\end{proof}

Finally, we will need the Cauchy--Davenport theorem. 

\begin{lemma}\label{lem:cauchydavenport}
Let $m \in \N$, let $p$ be a prime, and let $A \subset \Z_p$ be such that $m \cdot A \ne \Z_p$. Then 
$$|m \cdot A| \ge  m|A| - m + 1.$$
\end{lemma}


We are now ready to prove Hal\'{a}sz's Anticoncentration Lemma.

\begin{proof}[Proof of Lemma~\ref{lem:halasz}]
Let $v \in \Z_p^n \setminus \{0\}$, and let $1 \le t \le \ell \le 2^{-6} |v|$. We claim first that $|T_\ell(v)| < p$. To see this, let $a$ be a uniformly-chosen random element of $\Z_p$, and note that for each fixed $k \in \Z_p \setminus \{0\}$ we have $\Pr\big( \| a \cdot k / p \| \ge 1/4 \big) > 1/4$, and therefore
\begin{equation}\label{eq:easy:expectation}
\Ex\bigg[ \sum_{i=1}^n \left\| \frac{a \cdot v_i}{p} \right\|^2 \bigg] > \frac{|v|}{2^6}.
\end{equation}
Since $\ell \le 2^{-6}|v|$, it follows that there exists $k \in \Z_p$ with $k \not\in T_\ell(v)$, as claimed. 

Now, by Lemma \ref{lem:msumsetofT}, applied with $m := \big\lfloor \sqrt{\ell/t} \big\rfloor \ge \sqrt{\ell}/(2\sqrt{t})$, and by the definitions of $T_t(v)$ and $|v|$, we have
$$|m \cdot T_t(v)| \le |T_{m^2 t}(v)| \le |T_\ell(v)|.$$
By the Cauchy--Davenport theorem, it follows that $|m \cdot T_t(v)| \ge m \big( |T_t(v)| - 1 \big)$, and hence
$$|T_t(v)| \le 1 + \dfrac{|T_\ell(v)|}{m} \le 1 + 2 \sqrt{\dfrac{t}{\ell}} \cdot |T_\ell(v)|.$$
Combining this with Lemma~\ref{lem:H:second:bound}, we obtain
$$\rho (v) \le \dfrac{3}{p} + \dfrac{2e}{p} \cdot \frac{|T_\ell(v)|}{\sqrt{\ell}} \sum_{t = 1}^{\lceil \ell \rceil} \sqrt{t} e^{-t} + e^{-\ell} \le \dfrac{3}{p} + \frac{4 |T_\ell(v)|}{p \sqrt{\ell}} + e^{-\ell},$$
as claimed.
\end{proof}

}

\end{document}